\documentclass[a4paper,twoside,reqno,11pt,tbtags]{article}

\usepackage{amssymb,amsmath,amsthm,mathrsfs}
\usepackage{verbatim, graphicx}
\usepackage{natbib}
\bibpunct{(}{)}{;}{a}{,}{,}

\numberwithin{equation}{section}

\theoremstyle{plain}
\newtheorem{prop}{Proposition}[section]
\newtheorem{thm}[prop]{Theorem}

\theoremstyle{definition}
\newtheorem*{defi}{Definition}

\newtheorem{rem}[prop]{Remark}
\newtheorem*{nonumrem}{Remark}
\newtheorem*{example}{Example}

\DeclareMathOperator{\poisson}{Po}

\DeclareMathOperator{\binomial}{Bi}

\DeclareMathOperator{\diam}{diam}

\DeclareMathOperator*{\argmax}{arg \, max}

\newcommand{\Eta}{\mathrm{H}}

\newcommand{\inlawto}{\, \stackrel{\mathscr{D}}{\longrightarrow} \,}

\newcommand{\eqinlaw}{\, \stackrel{\mathscr{D}}{=} \,}
\newcommand{\bs}[1]{\boldsymbol{#1}}

\newcommand{\ssst}{\scriptscriptstyle}

\newcommand{\DD}{\mathbb{D}}
\newcommand{\EE}{\mathbb{E}}

\newcommand{\NN}{\mathbb{N}}
\newcommand{\PP}{\mathbb{P}}

\newcommand{\RR}{\mathbb{R}}
\newcommand{\ZZ}{\mathbb{Z}}

\newcommand{\Rplus}{\mathbb{R}_{+}}

\newcommand{\Zplus}{\mathbb{Z}_{+}}

\newcommand{\Leb}{\text{\textsl{Leb}}}

\newcommand{\dtv}{d_{TV}}
\newcommand{\dw}{d_{W}}

\newcommand{\dr}{d_{R}}
\newcommand{\drw}{d_{RW}}

\newcommand{\dzero}{d_{0}}
\newcommand{\done}{d_{1}}
\newcommand{\donebar}{\bar{d}_1}
\newcommand{\donep}{d_{1}^{(p)}}

\newcommand{\doneone}{d_{1}^{(1)}}

\newcommand{\doneinfty}{d_{1}^{{\ssst (} \infty {\ssst )}}}
\newcommand{\dtwo}{d_{2}}
\newcommand{\dtwobar}{\bar{d}_2}

\newcommand{\fwstar}{\mathcal{F}_{W}^{*}}

\newcommand{\ftwo}{\mathcal{F}_{2}}
\newcommand{\ftwobar}{\overline{\mathcal{F}}_{2}}

\newcommand{\abs}[1]{\lvert #1 \rvert}
\newcommand{\bigabs}[1]{\bigl| #1 \bigr|}
\newcommand{\Bigabs}[1]{\Bigl| #1 \Bigr|}
\newcommand{\biggabs}[1]{\biggl| #1 \biggr|}

\newcommand{\norm}[1]{\lVert #1 \rVert}

\newcommand{\ubar}{\bar{u}}
\newcommand{\xibar}{\bar{\xi}}
\newcommand{\etabar}{\bar{\eta}}

\newcommand{\tf}{\tilde{f}}
\newcommand{\tg}{\tilde{g}}

\newcommand{\tM}{\tilde{M}}
\newcommand{\tN}{\tilde{N}}

\newcommand{\tX}{\tilde{X}}
\newcommand{\tY}{\tilde{Y}}

\newcommand{\teta}{\tilde{\eta}}

\newcommand{\mcb}{\mathcal{B}}

\newcommand{\mca}{\mathcal{A}}

\newcommand{\mcn}{\mathcal{N}}

\newcommand{\mcx}{\mathcal{X}}
\newcommand{\mcy}{\mathcal{Y}}
\newcommand{\msl}{\mathscr{L}}

\newcommand{\mfn}{\mathfrak{N}}
\newcommand{\mfp}{\mathfrak{P}}

\newcommand{\tsum}{\textstyle{\sum}}

\newcommand{\bl}{\bs{\lambda}}
\newcommand{\bm}{\bs{\mu}}
\newcommand{\bn}{\bs{\nu}}

\newcommand{\RD}{\RR^D}

\newcommand{\nin}{\noindent}
\newcommand{\arrayeq}{\hspace*{-0.55em}=\hspace*{-0.55em}}
\newcommand{\arrayleq}{\hspace*{-0.55em}\leq\hspace*{-0.55em}}
\newcommand{\arraygeq}{\hspace*{-0.55em}\geq\hspace*{-0.55em}}

\renewcommand{\theprop}{\arabic{section}.\Alph{prop}}

\newenvironment{expla}{\comment}{\endcomment}
\relax

\newcommand{\mean}{\EE}
\newcommand{\prob}{\PP}
\newcommand{\Ref}{\eqref}

\begin{document}

\title{A new metric between distributions of point processes} 
\author{
Dominic Schuhmacher\footnote{Postal address: School of Mathematics and Statistics, The University of Western Australia, Crawley WA 6009, Australia. E-mail: dominic@maths.uwa.edu.au (corresponding author)}\\
The University of Western Australia\\[1mm]
and\\[1mm]
Aihua Xia\footnote{Postal address: Department of Mathematics and Statistics, The University of Melbourne, Parkville VIC 3010, Australia. E-mail: xia@ms.unimelb.edu.au}\\
The University of Melbourne\\[-3mm]
\hspace*{2mm}
}

\date{Version of 21 August, 2007}
\maketitle

\begin{abstract}
   Most metrics between finite point measures currently used in the literature have the flaw that 
   they do not treat differing total masses in an adequate manner for applications.
   This paper introduces a new metric $\donebar$ that combines positional differences of points
   under a closest match with the relative difference in total mass in a way that fixes this flaw.
   A comprehensive collection of theoretical results about $\donebar$ and its induced Wasserstein metric 
   $\dtwobar$ for point process distributions are given, including 
   examples of useful $\donebar$-Lipschitz continuous functions,
   $\dtwobar$ upper bounds for Poisson process approximation, and $\dtwobar$ upper and lower bounds between 
   distributions of point processes of i.i.d.\ points. Furthermore, we present a statistical test for 
   multiple point pattern data that demonstrates the potential of $\donebar$ in applications.\\[1mm]
\end{abstract}

\section{Introduction} \label{sec:intro}

The two metrics most widely used on the space $\mfn$ of finite point measures on a compact metric space $(\mcx, \dzero)$ are the Prohorov metric $\varrho$ and the metric $\done$ that was introduced in \citet{bb92}. We use $\delta_x$ to stand for the Dirac measure at $x$.
For $\xi = \sum_{i=1}^m \delta_{x_i}, \eta = \sum_{i=1}^n \delta_{y_i} \in \mfn$, and $\dzero \leq 1$ the metric $\done$ is given by
\begin{equation} \label{eq:donedef}
   \done(\xi,\eta) := \min_{\pi \in \Pi_n} \frac{1}{n} \sum_{i=1}^{n} \dzero(x_i, y_{\pi(i)})
\end{equation}
if $m=n \geq 1$ and $\done(\xi,\eta) := 1$ if $m \neq n$, where $\Pi_n$ denotes the set of permutations of $\{1,2,\ldots,n\}$. The gap between $\done =: \doneone$ and $\varrho \wedge 1 =: \doneinfty$ can be bridged by metrics $\donep$ where the average in~\eqref{eq:donedef} is replaced by a general $p$-th order average \citep[see][]{schumi5}.

All of these metrics are good choices from a theoretical point of view, because they metrize the natural vague topology on $\mfn$. Furthermore, especially $d_1$ has been highly successful as an underlying metric for defining a Wasserstein metric $\dtwo$ between point process distributions: letting $\ftwo := \bigl\{ f: \mfn \to [0,1] ; \, \abs{f(\xi)-f(\eta)} \leq \done(\xi,\eta) \text{ for all } \xi, \eta \in  \mfn \bigr\}$, we set
\begin{equation} \label{eq:dtwodef}
   \dtwo(P,Q) := \sup_{f \in \ftwo} \biggabs{ \int f \; dP - \int f \; dQ}
\end{equation}
for any two probability measures $P$ and $Q$ on $\mfn$. Numerous useful upper bounds in this metric
have been obtained; included amongst them are the results of \citet{bb92}, \citet{brownxia95}, \citet{brownxia01}, \citet{barbourmansson02}, \citet{chenxia04}, \citet{schumi4}, and \citet{schumi5}, which for the most part assume that one of the probability measures involved is a Poisson (or compound Poisson) process distribution.
Such estimates can be used to compare the distributions of point pattern statistics $S(\Xi)$, where $S \in \ftwo$, for different underlying point process models, since the Wasserstein distance $\dw \bigl( \msl(S(\Xi)),\msl(S(\Xi')) \bigr)$ (see pp.~254--255 of~\citet{bhj92}) is easily seen to be bounded by $\dtwo \bigl( \msl(\Xi), \msl(\Xi') \bigr)$. For a concrete example where this was exploited, see \citet[Section~3.2]{schumi1}.

However, there are certain limitations with respect to the practical applications of the metric~$\done$ (as well as of the other metrics between point measures that were mentioned), which are mainly due to the fact that $\done(\xi,\eta)$ is always set to the maximal distance $1$ if the total numbers of points of the point patterns $\xi$ and $\eta$ disagree. Such crude treatment results in a metric that does usually not reflect very well our intuitive idea of two point patterns being ``far apart'' from one another if the cardinalities of the point patterns are different, as can be seen from the extreme case illustrated in Figure~\ref{fig:spotthediff}. This flaw is, in our opinion, the main reason why such metrics have not been taken up in more application-oriented fields, such as spatial statistics. 

In the present article we introduce a new metric $\donebar$, which refines the metric $\done$ in the sense that $\donebar(\xi,\eta) = \done(\xi,\eta)$ if the cardinalities of the two point patterns $\xi$ and $\eta$ agree, but $\donebar(\xi,\eta)$ can take general values in $(0,1]$ if the cardinalities disagree. In particular, $\donebar$ assigns a large distance if the difference of the numbers of points is large compared to the total number of points in the point pattern with more points and it takes the quality of point matchings into account even if the total numbers are not the same.

While $\donebar$ is a slightly weaker metric than $\done$, it still metrizes the same topology as $\done$, and its induced Wasserstein metric $\dtwobar$ still metrizes convergence in distribution of point processes and provides an upper bound for the Wasserstein distance $\dw \bigl( \msl(S(\Xi)),\msl(S(\Xi') \bigr)$ for many of the useful point pattern statistics $S$ that $\dtwo$ does. As far as Poisson process approximation is concerned, we are able to obtain better bounds in the $\dtwobar$-metric than in the stronger $\dtwo$-metric for a wide range of situations. We furthermore present a simulation study that assesses the powers of certain tests based on $\donebar$ and demonstrates its usefulness in spatial statistics.

\begin{figure}[t]
\hspace*{2mm}
\vspace*{-9mm}

\begin{center}
\includegraphics[width=40mm,angle=-90]{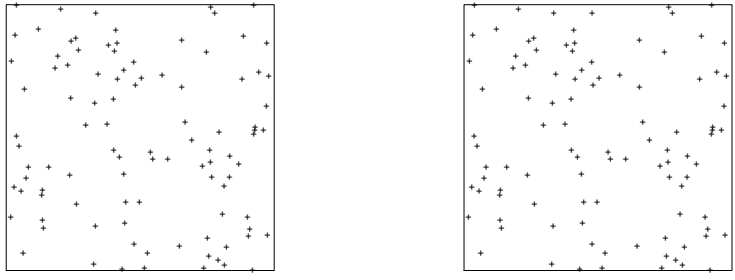}
\vspace*{-3mm}

\caption{\small The left is a realization of 99 independent and uniformly distributed points and the right is the same as the left except an additional point is added. Intuitively, we would say both point patterns are very similar. However, the $\done$-distance between the two is maximal, whereas the $\donebar$-distance is only 0.01 (out of a possible range of [0,1])}
\label{fig:spotthediff}
\end{center}
\vspace*{-8mm}

\hspace*{2mm}
\end{figure}

\section{Definition and elementary properties} \label{sec:defi}

Let $(\mcx, \dzero)$ be a compact metric space with $\dzero \leq 1$, on which we always consider the Borel $\sigma$-algebra $\mcb$. Denote the space of all finite point measures on $\mcx$ by $\mfn$ and equip it as usual with the vague topology and the $\sigma$-algebra $\mcn$ generated by this topology, which is the smallest $\sigma$-algebra that renders the point counts on measurable sets measurable (see \citealp{kallenberg86}, Section~1.1, Lemma~4.1, and Section~15.7). Recall that a point process is just a random element of $\mfn$. 
\begin{defi}
   Let $\donebar$ be the symmetric map $\mfn^2 \to \Rplus$ that is given by
   \begin{equation*}
      \donebar(\xi,\eta) := \frac{1}{n} \biggl( \min_{\pi \in \Pi_n} \sum_{i=1}^{m}
         \dzero(x_i, y_{\pi(i)}) + (n-m) \biggr)
   \end{equation*}
   for $\xi = \sum_{i=1}^m \delta_{x_i}, \eta = \sum_{j=1}^n \delta_{y_j} \in \mfn$ with
   $n \geq \max(m,1)$, and $\donebar(0,0):=0$.   
\end{defi}
\nin
In essence, we arrange for $\xi$ and $\eta$ to have the same number of points by introducing extra points located at distance $1$ from $\mcx$, and then take the average distance between the points under a closest match (which is the $\done$-distance).

\begin{prop} \label{prop:donebarismetric}
   The map $\donebar$ is a metric that is bounded by $1$.
\end{prop}
\nin
The proof of this proposition, as well as further proofs that are of a more technical nature and would otherwise disrupt the flow of the main text can be found in the appendix.
It is convenient to introduce the ``relative difference metric'' $\dr$ on $\Zplus$, which is given by
$\dr(m,n) := \abs{m-n} \big/ \max(m,n)$ for $\max(m,n) > 0$. The triangle inequality for $\dr$ follows immediately from the triangle inequality for $\donebar$, because we have $\dr(m,n) = \donebar \bigl( m \delta_x, n \delta_x \bigr)$. 
\begin{prop} \label{prop:donebarprops}
   The following statements about $\donebar$ hold.
   \begin{enumerate}
      \item $\dr(\abs{\xi},\abs{\eta}) \leq \donebar(\xi,\eta) \leq \done(\xi,\eta)$ for all $\xi, \eta \in
      \mfn$;
      \item $\donebar$ metrizes the vague (=weak) topology on $\mfn$;
      \item The metric space $(\mfn, \donebar)$ is locally compact, complete, and separable. 
   \end{enumerate}
\end{prop}

We next define the metric $\dtwobar$ on the space $\mfp(\mfn)$ of probability distributions on $(\mfn, \mcn)$ just as the Wasserstein metric with respect to $\donebar$.
\begin{defi}
Let $\ftwobar := \{ f : \mfn \to [0,1] \, ; \; \abs{f(\xi) - f(\eta)} \leq \donebar(\xi,\eta) \text{ for all $\xi, \eta \in \mfn$} \}$. Set then
$$
   \dtwobar(P,Q) := \sup_{f \in \ftwobar} \biggabs{\int_{\mfn} f \; dP - \int_{\mfn} f \; dQ}
$$
for $P,Q \in \mfp(\mfn)$.
\end{defi}
Since this is exactly the Wasserstein construction (the fact that we restrict the functions in $\ftwobar$ to be $[0,1]$-valued has no influence on the supremum, because the underlying $\donebar$-metric is bounded by $1$), it is clear that $\dtwobar$ is a metric that is obviously bounded by $1$, and we can easily derive basic properties. For two probability distributions $\mu$ and $\nu$ on $\Zplus$, write
$\drw(\mu,\nu) := \min_{M \sim \mu, N \sim \nu} \EE \dr(M,N)$, which is the Wasserstein distance with respect to $\dr$ (compare property~(i) below).
\begin{prop} \label{prop:dtwobarprops}
   The metric $\dtwobar$ satisfies
   \begin{enumerate}
      \item $\dtwobar \bigl( P, Q \bigr) = \min_{\substack{\Xi \sim P \\ \Eta \sim Q}} \: \EE \,
			\donebar(\Xi,\Eta)$ for all $P, Q \in \mfp(\mfn)$;
      \item $\drw \bigl( \msl(\abs{\Xi}), \msl(\abs{\Eta}) \bigr) \leq \dtwobar \bigl( \msl(\Xi),
         \msl(\Eta) \bigr) \leq \dtwo \bigl( \msl(\Xi), \msl(\Eta) \bigr)$ for any point processes
         $\Xi$ and $\Eta$;
      \item $\dtwobar$ metrizes the weak topology on $\mfp(\mfn)$, so that $\Xi_n \inlawto \Xi$ iff
         $\dtwobar \bigl( \msl(\Xi_n), \msl(\Xi) \bigr) \longrightarrow 0$.
   \end{enumerate}
\end{prop}

\section{Lipschitz continuous functions}

By the definition of $\dtwobar$, upper bounds for a distance $\dtwobar \bigl( \msl(\Xi), \msl(\Eta) \bigr)$ also bound the difference $\bigabs{\EE f(\Xi) - \EE f(\Eta)}$  for any $f \in \ftwobar$. It is thus of considerable interest for the application of estimates such as those obtained in Section~\ref{sec:dests} to have a certain supply of ``meaningful'' $\donebar$-Lipschitz continuous statistics of point patterns (where we do not worry too much about the Lipschitz constant as it will only appear as an additional factor in the upper bound).

For the $\done$-metric, a selection of such statistics was given in Section~10.2 of~\citet{bhj92} and in Subsection~3.3.1 of \citet{mythesis}. Since $\donebar$ is in general strictly smaller than $\done$, we cannot reasonably expect all of these functions to lie in $\ftwobar$. However, we are able to recover many of the most important examples, which is illustrated by the two propositions below. This is mainly due to the fact that these functions take all the points in the pattern into account without fundamentally distinguishing how many there are, which is a situation where a $\done$-Lipschitz condition typically provides too much room in the upper bound. 

Our first proposition concerns certain $U$-statistics with Lipschitz continuous kernels (the former are usually considered for a fixed number of points, but the extension is obvious). See \cite{lee90} for detailed results about such statistics. 
\begin{prop}  \label{prop:ftwopfunctions}
Suppose that $\mcy \supset \mcx$ and extend the metric $\dzero$ to $\mcy$ in such a way that it is still bounded by $1$. Fix $l \in \NN := \{1,2,\ldots\}$ and  write $\mfn_{l+} := \{ \xi \in \mfn ; \, \abs{\xi} \geq l \}$. Let $K : \mcy^l \to [0,1]$ be a symmetric function that satisfies
\begin{enumerate}
   \item $\bigabs{ K(u_1,\ldots,u_l) - K(v_1,\ldots,v_l)} \leq \frac{1}{l} \sum_{i=1}^l \dzero(u_i,v_i)$
      for all $u_1, \ldots, u_l, v_1, \ldots, v_l \in \mcy$;
   \item for every $N \in \NN$ there are $\ubar_1, \ldots, \ubar_N \in \mcy$ such that for any $k \in \{
      1, \ldots, l \}$ and any selection $1 \leq i_1 < \ldots < i_k \leq N$ of $k$ indices
      \begin{equation*} 
         K(\ubar_{i_1}, \ubar_{i_2}, \ldots, \ubar_{i_k}, u_{k+1}, u_{k+2}, \ldots, u_l) \geq K(u_1, u_2,
         \ldots, u_k, u_{k+1}, u_{k+2}, \ldots, u_l) 
      \end{equation*}
      for all $u_1, u_2, \ldots, u_l \in \mcx$; 
   \item for every $k \in \{ 1, \ldots, l \}$ we have 
      \begin{equation*}
         K(u_1, u_1, \ldots, u_1, u_{k+1}, u_{k+2}, \ldots, u_l) \leq K(u_1, u_2, \ldots, u_k, u_{k+1}, 
         u_{k+2}, \ldots, u_l)
      \end{equation*}
      for all $u_1, u_2, \ldots, u_l \in \mcx$.      
\end{enumerate}
Define $f: \mfn_{l+} \to [0,1]$ by
\begin{equation} \label{eq:ftwopfunctions}
   f(\xi) := \frac{1}{{m \choose l}} \sum_{1 \leq i_1 < i_2 < \ldots < i_l \leq m} 
      K(x_{i_1}, x_{i_2}, \ldots, x_{i_l})
\end{equation}
for $\xi = \sum_{i=1}^m \delta_{x_i} \in \mfn$ with $m \geq l$.
Then there exists an extension $F$ of $f$ to the whole of $\mfn$ such that \raisebox{0pt}[11pt][0pt]{$F \in \ftwobar$.}
\end{prop}
\nin
One possible choice for the function $K$ in the above result is half the interpoint distance, i.e. $K(u_1,u_2) = \frac{1}{2} \dzero(u_1,u_2)$ for all $u_1,u_2 \in \mcx$. If $\mcx \subset \RD =: \mcy$ for some $D \in \NN$ and $\dzero(x,y) =  \abs{x-y} \wedge 1$ for all $x,y \in \RD$, we can consider more generally the diameter of the minimal bounding ball, defining
\begin{equation*}
   K(u_1,\ldots,u_l) := \frac{1}{l} \min \bigl\{ \diam_0(B); \, B \subset \RD \text{ closed Euclidean 
   ball with } u_1, \dots, u_l \in B \bigr\}
\end{equation*}
for $l \geq 2$ and $u_1, \ldots, u_l \in \RD$, where $\diam_0(B) := \sup \{ \dzero(x,y); x,y \in B \}$. It can be shown that this yields again a function $K$ that satisfies (i)--(iii).

The second proposition looks at the average nearest neighbor distance in a finite point pattern on $\RD$. This statistic gives important information about the amount of clustering in the pattern.
\begin{prop} \label{prop:nndist}
Let $\mcx \subset \RD$, and $\dzero(x,y) = \abs{x-y} \wedge 1$ for all $x,y \in \RD$. Define the function $f: \mfn_{2+} \to [0,1]$ by
\begin{equation*}
   f(\xi) := \frac{1}{m} \sum_{i=1}^m \min_{\substack{j \in \{1,\ldots,m\} \\ j \neq i}}
      \dzero(x_i,x_j)
\end{equation*}
for $\xi = \sum_{i=1}^m \delta_{x_i} \in \mfn$ with $m \geq 2$. Then there exists an extension $F$ of $f$ to the whole of $\mfn$ that is $\donebar$-Lipschitz continuous with constant~$\tau_D+1$, where $\tau_D$ denotes the kissing number in $D$ dimensions (i.e.\ the maximal number of unit balls that can touch a unit ball in $\RD$ without producing any overlaps of the interiors; see \cite{cs99}, Section~1.2, for details).
\end{prop}
\begin{proof}[Proof of Proposition~\ref{prop:ftwopfunctions}]
Fix a point $x_0 \in \mcx$ and define $F(\xi') := f \bigl( \xi' + (l-\abs{\xi'})^{+} \delta_{x_0} \bigr)$ for every $\xi' \in \mfn$. It suffices to show that $\abs{f(\xi) - f(\eta)} \leq \donebar(\xi,\eta)$ for $\xi, \eta \in \mfn$ with $\abs{\xi}, \abs{\eta} \geq l$, because this implies that
\begin{equation*}
\begin{split}
   \bigabs{F(\xi') - F(\eta')} &= \bigabs{ f \bigl( \xi' + (l-\abs{\xi'})^{+} \delta_{x_0} \bigr) - f \bigl(
      \eta' + (l-\abs{\eta'})^{+} \delta_{x_0} \bigr) } \\
   &\leq \donebar \bigl( \xi' + (l-\abs{\xi'})^{+} \delta_{x_0}, \, \eta' + (l-\abs{\eta'})^{+} \delta_{x_0}
      \bigr) \\[0.5mm]
   &\leq \donebar (\xi', \eta') 
\end{split}
\end{equation*}
for every $\xi',\eta' \in \mfn$.
Let then $\xi=\sum_{i=1}^m \delta_{x_i}$ and $\eta=\sum_{i=1}^n \delta_{y_i}$, where $m, n \geq l$ and without loss of generality $m \leq n$ (because of the symmetry of the inequality that we would like to show). We add $n-m$ points $x_{m+1}, \dots, x_n$ to $\xi$ in one of the following two ways depending on whether $f(\xi) \geq f(\eta)$ or $f(\xi) < f(\eta)$, and call the result $\xibar := \sum_{i=1}^n \delta_{x_i}$.

If $f(\xi) \geq f(\eta)$, let $x_{m+r} := \bar{u}_r$, $1 \leq r \leq n-m$, for points $\bar{u}_1, \dots, \bar{u}_{n-m}$ chosen as in assumption (ii) with $N = n-m$. It follows that
\begin{equation} \label{eq:increaseU}
\begin{split}
   f(\xibar) &= \frac{1}{{n \choose l}} \sum_{1 \leq i_1 < \ldots < i_l \leq n} 
      K(x_{i_1}, \ldots, x_{i_l}) \\
   &= \frac{1}{\sum_{j=0}^l {m \choose j}{n-m \choose l-j}} \: \sum_{j=0}^l \sum_{\substack{1 \leq 
      i_1 < \ldots < i_j \leq m \\[1.5pt] m+1 \leq i_{j+1} < \ldots < i_l \leq n}} K(x_{i_1}, \ldots, 
      x_{i_l}) \\
   &\geq \frac{1}{{m \choose l}} \sum_{1 \leq i_1 < \ldots < i_l \leq m} 
      K(x_{i_1}, \ldots, x_{i_l}) = f(\xi).
\end{split}
\end{equation}
The inequality is a consequence of the fact that $\bigl( \sum_{j=0}^l a_j \bigr) \big/ \bigl( \sum_{j=0}^l b_j \bigr) \geq a_l/b_l$ if $a_j/b_j \geq a_l/b_l$ for every $j$; 
and the latter condition holds because for $\max(0,l-n+m) \leq j \leq l-1$ (since $a_j=b_j=0$ if $j < l-n+m$, these pairs can be ignored altogether),
\begin{equation*}
\begin{split}
   &\frac{1}{{m \choose j}{n-m \choose l-j}} \sum_{\substack{1 \leq 
      i_1 < \ldots < i_j \leq m \\[1.5pt] m+1 \leq i_{j+1} < \ldots < i_l \leq n}} K(x_{i_1}, \ldots, 
      x_{i_l}) \\
   &\hspace*{7mm} \geq \frac{1}{{m \choose j}{n-m \choose l-j}} \sum_{\substack{1 \leq 
      i_1 < \ldots < i_j \leq m \\[1.5pt] m+1 \leq i_{j+1} < \ldots < i_l \leq n}} \frac{1}{{m-j \choose
      l-j}} \sum_{\substack{1 \leq r_{j+1} < \ldots < r_l \leq m \\ \{r_{j+1},\ldots,r_l\} \cap 
      \{i_1,\ldots,i_j\} = \emptyset}} K(x_{i_1}, \ldots, x_{i_j}, x_{r_{j+1}}, \ldots, x_{r_l}) \\
   &\hspace*{7mm} = \frac{1}{{m \choose j}} \frac{1}{{m-j \choose l-j}} \sum_{1 \leq 
      i_1 < \ldots < i_j \leq m} \; \sum_{\substack{1 \leq r_{j+1} < \ldots < r_l \leq m \\ 
      \{r_{j+1},\ldots,r_l\} \cap 
      \{i_1,\ldots,i_j\} = \emptyset}} K(x_{i_1}, \ldots, x_{i_j}, x_{r_{j+1}}, \ldots, x_{r_l}) \\[-1.5mm]
   &\hspace*{7mm} = \frac{{l \choose j}}{{m \choose j}{m-j \choose l-j}} \sum_{1 \leq i_1 < \ldots < i_l 
      \leq m} K(x_{i_1}, \ldots, x_{i_l}) \\[1mm]
   &\hspace*{7mm} = \frac{1}{{m \choose l}} \sum_{1 \leq i_1 < \ldots < i_l \leq m}
      K(x_{i_1}, \ldots, x_{i_l}),
\end{split}   
\end{equation*}
where the inequality follows by assumption (ii) and the symmetry of $K$. 

If on the other hand $f(\xi) < f(\eta)$, let $x_{m+r} := x_1$, $1 \leq r \leq n-m$. It follows in exactly the same way as for the first case, only this time with ``$\geq$'' replaced by ``$\leq$'' and using assumption~(iii) instead of assumption~(ii), that $f(\xibar) \leq f(\xi)$.

In total, we thus obtain
\begin{equation*}
   \bigabs{f(\xi) - f(\eta)} \leq \bigabs{f(\xibar) - f(\eta)} \leq \done(\xibar,\eta) = 
   \donebar(\xibar,\eta) \leq \donebar(\xi,\eta),
\end{equation*}
where the second inequality follows from the $\done$-Lipschitz continuity of the functions considered in Proposition~2.A of \cite{schumi5}.
\end{proof}

\begin{proof}[Proof of Proposition~\ref{prop:nndist}]
Fix arbitrary $\alpha_0, \alpha_1 \in [0,1]$ and define $F(\xi) := \alpha_i$ if $\abs{\xi} = i \in \{0,1\}$ and $F(\xi) = f(\xi)$ otherwise. Let $\xi = \sum_{i=1}^m \delta_{x_i}$ and $\eta = \sum_{i=1}^n \delta_{y_i}$, where without loss of generality we assume $m \leq n$. Since $\bigabs{F(\xi) - F(\eta)} \leq 1 \leq (\tau_D+1) \frac12 \leq (\tau_D+1) \donebar(\xi,\eta)$ if $m \in \{0,1\}$ and $n>m$, the Lipschitz inequality remains to be shown for $n \geq m \geq 2$ only.

As before, we bring the cardinalities to the same level. Let $\xibar := \sum_{i=1}^n \delta_{x_i}$, where the points $x_{m+1}, \ldots, x_n$ are chosen in the following way. If $f(\xi) \geq f(\eta)$, let $x_{m+1}, \ldots, x_n$ be arbitrary pairwise distinct points in $\RD$ that are at $\dzero$-distance $1$ from each other and from $\mcx$. Hence $f(\xibar) \geq f(\xi)$ because for each of the added points the distance to its nearest neighbor is one, which is maximal.
If on the other hand $f(\xi) < f(\eta)$, let $x_{m+1} := \ldots := x_n := x_1$, whence it is immediately clear that $f(\xibar) \leq f(\xi)$ because for each of the added points the distance to its nearest neighbor is zero. 

In total, we obtain
\begin{equation*}
   \bigabs{f(\xi) - f(\eta)} \leq \bigabs{f(\xibar) - f(\eta)} \leq (\tau_D+1) \done(\xibar,\eta) = 
   (\tau_D+1) \donebar(\xibar,\eta) \leq (\tau_D+1) \donebar(\xi,\eta),
\end{equation*}
where the second inequality follows from the $\done$-Lipschitz continuity of the average nearest neighbor distance considered in Proposition~2.C of \cite{schumi5}.
\end{proof}

\section{Distance estimates in $\bs{\dtwobar}$} \label{sec:dests}

In this section we present upper bounds for some essential $\dtwobar$-distances, which all clearly improve on the bounds that are available for the corresponding $\dtwo$-distances. However, the improvement in general results is not always as much as one would hope it to be, and it seems that considerably better bounds can be obtained by a more specialized treatment (see for example Subsection~\ref{ssec:BePo}). 

\subsection{Poisson process approximation of a general point process} \label{ssec:genPo}

Using the fact that
\begin{equation}
\mca h(\xi)=\int_\mcx[h(\xi+\delta_\alpha)-h(\xi)]\;\bl(d\alpha)+\int_\mcx[h(\xi-\delta_\alpha)-h(\xi)]\;\xi(d\alpha),\ \xi\in\mfn,
\label{generator1}
\end{equation}
is the generator of the spatial immigration-death process whose steady state distribution is the Poisson process with expectation measure $\bl$, \citet{bb92} establish the Stein identity for Poisson process approximation as
\begin{equation}
\mca h(\xi)=f(\xi)-\poisson(\bl)(f)
\label{Steinidentity1}
\end{equation}
for suitable test functions $f$ on $\mfn$. The solution for (\ref{Steinidentity1}) is given by
\begin{equation}
h_f(\xi)=-\int_0^\infty[\mean f(\ZZ_\xi(t))-\poisson(\bl)(f)]\;dt,
\label{solution1}
\end{equation}
where $\ZZ_\xi$ is an immigration-death process with generator $\mca$ and initial point pattern $\ZZ_\xi(0)=\xi$. 

Using \Ref{Steinidentity1} and different characteristics of point processes, we can establish various Poisson process approximation error bounds (see \citet{bb92}, \citet{bb92comp}, \citet{bbx98}, and \citet{chenxia04}).
To keep our text concise, we present here a slightly simplified version of the main result in \citet{chenxia04} only; it is an obvious exercise to apply our estimates \Ref{bounddelta1} and \Ref{bounddelta2} to get parallel results in the other articles mentioned above.

We assume that, for each $\alpha\in\mcx$, there is 
a Borel set $A_\alpha\subset\mcx$ such that $\alpha\in A_\alpha$ and
the mapping 
$$\mcx\times\mfn\rightarrow\mcx\times\mfn: (\alpha,\xi)\mapsto(\alpha,\xi|_{A_\alpha^c})
$$
is product measurable, where $\xi|_{A_\alpha^c}$ stands for the point pattern of $\xi$ restricted to $A_\alpha^c$ (\citealp{kallenberg86}, Section~1.1).
Such requirement can be ensured by $A=\{(x,y);\, y\in A_x,x\in\mcx\}$ measurable in $\mcx^2$
 (see \citealp{chenxia04}). We define, for any function $h$ on $\mfn$, that
\begin{equation*}
\begin{split}
\Delta h(\xi)&:=\sup_{\alpha\in\mcx}|h(\xi+\delta_\alpha)-h(\xi)|,\\
\Delta^2h(\xi)&:=\sup_{\eta-\xi\in\mfn,\hspace*{1.5pt}\alpha,\beta\in\mcx}|h(\eta+\delta_\alpha+\delta_\beta)-h(\eta+\delta_\alpha)-h(\eta+\delta_\beta)+h(\eta)|,\ \xi\in\mfn.
\end{split}
\end{equation*}
 
\begin{thm}[\citealp{chenxia04}]\label{chenxia04thm}
For each bounded measurable function $f:\mfn\rightarrow\RR_+$, let $h_f$ be the solution \Ref{solution1} of Equation~\Ref{Steinidentity1}. If $\Xi$ is a point process on $\mcx$ with expectation measure $\bl$, then
\begin{eqnarray*}
&&|\mean f(\Xi)-\poisson(\bl)(f)|\\[1mm]
&&\le\mean\int_{{\mcx}}
\Delta^2h_f(\Xi|_{A_\alpha^c})(\Xi(A_\alpha)-1)\; \Xi(d\alpha)\\
&&\ \ \ +\mean\int_{{\mcx}}\left|[h_f(\Xi|_{A_\alpha^c})-h_f(\Xi|_{A_\alpha^c}+\delta_\alpha)]
-[h_f(\Xi_\alpha|_{A_\alpha^c})-h_f(\Xi_\alpha|_{A_\alpha^c}+\delta_\alpha)]\right|\; \bl(d\alpha)\\
&&\ \ \ +\mean\int_{{\mcx}} 
        \Delta^2h_f(\Xi|_{A_\alpha^c})\Xi(A_\alpha)\; \bl(d\alpha),
\end{eqnarray*}
where $\Xi_\alpha$ is the Palm process of $\Xi$
at location $\alpha\in\mcx$ (\citealp{kallenberg86}, Chapter~10).
\end{thm}

\nin
The error bounds for Poisson process approximation like Theorem~\ref{chenxia04thm} (see \citet{bb92}, \citet{bb92comp}, \citet{bbx98}, and \citet{chenxia04} for full details) pivot on the estimates of $\Delta h_f$ and $\Delta^2h_f$. The following proposition summarizes these estimates for $\dtwobar$.
 
\begin{prop} Let
\begin{equation*}
\begin{split}
\Delta h(\xi;\alpha)&=h(\xi+\delta_\alpha)-h(\xi),\\
\Delta^2h(\xi;\alpha,\beta)&=h(\xi+\delta_\alpha+\delta_\beta)-h(\xi+\delta_\alpha)-h(\xi+\delta_\beta)+h(\xi),\ \xi\in\mfn,\ \alpha,\beta\in\mcx;
\end{split}
\end{equation*}
then for each $\donebar$-Lipschitz function $f$, we have
\begin{eqnarray}
|\Delta h_f(\xi;\alpha)|&\arrayleq& \min\left\{1,\frac{0.95+\ln^+\lambda}{\lambda},\frac{1-e^{-|\xi|\wedge\lambda}}{|\xi|\wedge\lambda}\right\},
	\label{bounddelta1}\\
|\Delta^2h_f(\xi;\alpha,\beta)|&\arrayleq& \min\left\{0.75,\frac{1}{|\xi|\wedge\lambda},\frac{1.09}{|\xi|+1}+\frac{1}{\lambda},\frac{2\ln\lambda}{\lambda}{\bf 1}_{\{\lambda\ge1.76\}}+0.75\,{\bf 1}_{\{\lambda<1.76\}}\right\},
	\label{bounddelta2}
\end{eqnarray}
where $\frac{1-e^{0}}{0}:=1$ and $\lambda=\bl(\mcx)$.
\end{prop}

\begin{proof} For convenience, we write $|\xi|=n$ and $|\ZZ_\xi(t)|=Z_\xi(t)$. Let $\tau_1$ and $\tau_2$ be independent exponential random variables with mean $1$ which are also independent of $\ZZ_\xi$; then one can write
\begin{equation*}
\begin{split}
   &\ZZ_{\xi+\delta_\alpha}(t)=\ZZ_\xi(t)+\delta_\alpha {\bf 1}_{\{\tau_1>t\}},\ \ZZ_{\xi+\delta_\beta}(t)=
      \ZZ_\xi(t)+\delta_\beta {\bf 1}_{\{\tau_2>t\}},\\
   &\hspace*{2.25em}\text{and } \, \ZZ_{\xi+\delta_\alpha+\delta_\beta}(t)=\ZZ_\xi(t)+\delta_\alpha
      {\bf 1}_{\{\tau_1>t\}}+\delta_\beta {\bf 1}_{\{\tau_2>t\}}.
\end{split}
\end{equation*}
Hence it follows from \Ref{solution1} and the $\donebar$-Lipschitz property of $f$ that
\begin{eqnarray}
\left|\Delta h_f(\xi;\alpha)\right|
&\arrayeq&\left|\int_0^\infty e^{-t}\mean[f(\ZZ_\xi(t)+\delta_\alpha)
	-f(\ZZ_\xi(t))]\;dt\right|\label{delta1proof1}\\
&\arrayleq&\int_0^\infty e^{-t}\mean\frac{1}{Z_\xi(t)+1}\;dt\label{delta1proof2}\\
&\arrayleq&\int_0^\infty e^{-t}\;dt=1.\nonumber
\end{eqnarray}
Also,
\begin{eqnarray}
&&\hspace*{-12mm}\left|\Delta^2h_f(\xi;\alpha,\beta)\right|\nonumber\\
&\arrayeq&\left|\int_0^\infty e^{-2t}\mean[f(\ZZ_\xi(t)+\delta_\alpha+\delta_\beta)-f(\ZZ_\xi(t)+\delta_\alpha)
	-f(\ZZ_\xi(t)+\delta_\beta)+f(\ZZ_\xi(t))]\;dt\right|\label{delta2proof1}\\
&\arrayleq&\int_0^\infty e^{-2t}\mean\left[\frac{1}{Z_\xi(t)+2}+\frac{1}{Z_\xi(t)+1}\right]\;dt\label{delta2proof2}\\
&\arrayleq&1.5\int_0^\infty e^{-2t}dt=0.75.\label{delta2proof3}
\end{eqnarray}
However, since $\ZZ_\xi$ has constant immigration rate $\bl$ and unit per capita death rate, it is possible to write
$$\ZZ_\xi(t)=\ZZ_\emptyset(t)+\DD_\xi(t),$$
where $\DD_\xi$ is a pure death process with unit per capita death rate independent of $\ZZ_\emptyset$. Direct verification gives that $Z_\emptyset(t)$ follows the Poisson distribution with mean $\lambda_t:=\lambda(1-e^{-t})$, while $|\DD_\xi(t)|$ follows $\binomial(|\xi|,e^{-t})$. Hence
\begin{equation}
\mean\frac{1}{Z_\xi(t)+1}\le\mean\frac{1}{Z_\emptyset(t)+1}=
\frac{1-e^{-\lambda_t}}{\lambda_t},\label{delta1proof4}
\end{equation}
\begin{eqnarray}
\mean\frac{1}{Z_\xi(t)+1}&\arrayeq&\int_0^1\mean x^{Z_\xi(t)}\;dx=\int_0^1[1-e^{-t}(1-x)]^ne^{-\lambda_t(1-x)}\;dx\nonumber\\
&\arrayleq& \int_0^1e^{-(ne^{-t}+\lambda_t)(1-x)}\;dx\le \int_0^1e^{-(n\wedge\lambda)(1-x)}\;dx=\frac{1-e^{-n\wedge\lambda}}{n\wedge\lambda},\label{delta2proof5}
\end{eqnarray}
and similarly,
\begin{eqnarray}
\mean\frac{1}{Z_\xi(t)+2}&\arrayeq&\int_0^1\mean x^{Z_\xi(t)+1}\;dx=\int_0^1x[1-e^{-t}(1-x)]^ne^{-\lambda_t(1-x)}\;dx\nonumber\\
&\arrayleq& \int_0^1xe^{-(n\wedge\lambda)(1-x)}\;dx=\frac{1}{n\wedge \lambda}
	-\frac{1}{(n\wedge\lambda)^2}(1-e^{-n\wedge\lambda}).\label{delta2proof6}
\end{eqnarray}
The claim
\begin{equation}
\left|\Delta h_f(\xi;\alpha)\right|\le \frac{0.95+\ln^+\lambda}{\lambda}
\label{delta1proof5}
\end{equation}
is obvious for $\lambda<0.95$ as the right hand side is already greater than 1, so it remains to show \Ref{delta1proof5} for $\lambda\ge 0.95$.
Combining \Ref{delta1proof2} and \Ref{delta1proof4}, with $s=1-e^{-t}$, we get 
\begin{eqnarray*}
\left|\Delta h_f(\xi;\alpha)\right|&\arrayleq&\int_0^\infty e^{-t}\frac{1-e^{-\lambda_t}}{\lambda_t}\;dt=\int_0^1\frac{1-e^{-\lambda s}}{\lambda s}\;ds\le\frac{1}{\lambda}\left(\frac{e^{-\lambda}}{\lambda}+\ln\lambda+\gamma\right),
\end{eqnarray*}
where $\gamma$ is the Euler constant and the last inequality is due to items 5.1.39 and 5.1.19 of \citet{abrastegun72}. For $0.95\le \lambda \le 1$, $\frac{e^{-\lambda}}{\lambda}+\ln\lambda+\gamma\le e^{-1}+\gamma<0.95$ since $\frac{e^{-\lambda}}{\lambda}+\ln\lambda+\gamma$ is increasing for $\lambda\ge 0.95$, and for $\lambda>1$, $\frac{e^{-\lambda}}{\lambda}+\gamma<e^{-1}+\gamma<0.95$ because the function $\frac{e^{-\lambda}}{\lambda}+\gamma$ is decreasing, completing the proof of \Ref{delta1proof5}. The last claim in \Ref{bounddelta1} is easily obtained from 
\Ref{delta1proof2} and \Ref{delta2proof5}.

We then apply \Ref{delta2proof5} and \Ref{delta2proof6} in \Ref{delta2proof2} to obtain
\begin{eqnarray}
\left|\Delta^2h_f(\xi;\alpha,\beta)\right|&\arrayleq&\frac{0.5}{n\wedge\lambda}\left\{2-e^{-n\wedge\lambda}
	-\frac{1}{n\wedge\lambda}(1-e^{-n\wedge\lambda})\right\}\label{delta2proof4}\\
&\arrayleq&\frac{1}{n\wedge\lambda}.\label{delta2proof7}
\end{eqnarray}
Now, we show that 
\begin{equation}
\left|\Delta^2h_f(\xi;\alpha,\beta)\right|\le \frac{1.09}{n+1}+\frac{1}{\lambda}.
\label{delta2proof8}
\end{equation}
As a matter of fact, by \Ref{delta2proof3} and \Ref{delta2proof7}, \Ref{delta2proof8} clearly holds for $n=0$ and $n\ge\lambda$, hence it remains to show \Ref{delta2proof8} for $1\le n<\lambda$. Using \Ref{delta2proof4}, it suffices to prove that
\begin{equation}
\frac{0.5(n+1)}{n}\left\{2-e^{-n}
	-\frac{1}{n}(1-e^{-n})\right\}\le 1.09.\label{delta2proof9}
\end{equation}
However, for $n\ge 12,$
$$\frac{0.5(n+1)}{n}\left\{2-e^{-n}
	-\frac{1}{n}(1-e^{-n})\right\}<\frac{n+1}{n}\le \frac{13}{12}<1.09$$
while for $1\le n\le 11$, one can verify \Ref{delta2proof9}  for each value of $n$. 

Finally, we prove 
\begin{equation}
|\Delta^2h_f(\xi;\alpha,\beta)|\le\frac{2\ln\lambda}{\lambda}{\bf 1}_{\{\lambda\ge1.76\}}+0.75\,{\bf 1}_{\{\lambda<1.76\}}.\label{delta2proof10}
\end{equation}
The claim \Ref{delta2proof10} is evident for $\lambda < 1.76$, so we assume $\lambda \geq 1.76$. On the other hand, if $Y$ follows $\poisson(\nu)$, then
\begin{equation*}
\mean\frac{1}{Y+2}=\mean\left\{\frac{1}{Y+1}-\frac{1}{(Y+1)(Y+2)}\right\}=
\frac{\nu-1+e^{-\nu}}{\nu^2}.
\end{equation*}
Therefore,
$$\mean\left\{\frac{1}{Z_\xi(t)+1}+\frac{1}{Z_\xi(t)+2}\right\}\le\mean\left\{\frac{1}{Z_\emptyset(t)+1}+\frac{1}{Z_\emptyset(t)+2}\right\}=\frac{2\lambda_t-1+(1-\lambda_t)e^{-\lambda_t}}{\lambda_t^2},
$$
which, together with \Ref{delta2proof2}, ensures that
\begin{eqnarray*}
\left|\Delta^2h_f(\xi;\alpha,\beta)\right|
&\arrayleq&\int_0^\infty e^{-2t}\frac{2\lambda_t-1+(1-\lambda_t)
	e^{-\lambda_t}}{\lambda_t^2}\;dt\\
&\arrayeq&\int_0^1(1-s)\frac{2\lambda s-1+(1-\lambda s)
	e^{-\lambda s}}{\lambda^2s^2}\;ds\\
&\arrayeq&-\frac{3}{\lambda}+\frac{2(1-e^{-\lambda})}{\lambda^2}+\left(\frac{2}{\lambda}+\frac{1}{\lambda^2}\right)\int_0^\lambda\frac{1-e^{-t}}{t}\;dt\\
&\arrayleq&-\frac{3}{\lambda}+\frac{2(1-e^{-\lambda})}{\lambda^2}+\left(\frac{2}{\lambda}+\frac{1}{\lambda^2}\right)\left(\frac{e^{-\lambda}}{\lambda}+\ln\lambda+\gamma\right)\\
&\arrayeq&-\frac{3}{\lambda}+\frac{2}{\lambda^2}+\frac{e^{-\lambda}}{\lambda^3}+\left(\frac{2}{\lambda}+\frac{1}{\lambda^2}\right)\left(\ln\lambda+\gamma\right)=:a(\lambda),
\end{eqnarray*}
where the first equality is by the change of variable $s=1-e^{-t}$ and the last inequality is from items 5.1.39 and 5.1.19 of \citet{abrastegun72}. 
Now, $b(\lambda):=a(\lambda)\lambda-2\ln\lambda$ is decreasing in $\lambda$ for $\lambda>1$ and $b(1.76)<0$,
which implies that $a(\lambda)\le \frac{2\ln\lambda}{\lambda}$ for $\lambda \ge 1.76$, completing the proof of \Ref{delta2proof10}.
\end{proof}

The following counter-example, adapted from \citet{brownxia95factors}, shows that the logarithmic factors in \Ref{bounddelta1} and \Ref{bounddelta2} can not be removed.

\begin{example} Let $\mcx=\{0,1\}$ with metric $d_0(x,y)=|x-y|$, let $\bl$ satisfy $\bl\{1\}=1$ and $\bl\{0\}=\lambda-1>0$, and define a $\donebar$-Lipschitz function on $\mfn$ as
$$f(\xi)=\left\{
\begin{array}{ll}
\frac{1}{|\xi|+1},&\mbox{ if }\xi\{1\}=0,\\
0,&\mbox{ if }\xi\{1\}>0.
\end{array}\right.$$
Using the fact that $\ZZ_\emptyset(t)\{0\}$ follows $\poisson\bigl((\lambda-1)(1-e^{-t})\bigr)$ and $\ZZ_\emptyset(t)\{1\}$ follows $\poisson(1-e^{-t})$, we have from \Ref{delta1proof1} and \Ref{delta2proof1} that, as $\lambda\rightarrow\infty$,
\begin{eqnarray*}
\left|\Delta h_f(\emptyset;1)\right|
&\arrayeq&\int_0^\infty e^{-t}\mean\frac{1}{\ZZ_\emptyset(t)\{0\}+1}\prob[\ZZ_\emptyset(t)\{1\}=0]\;dt\\
&\arrayeq&\int_0^\infty e^{-t}\frac{1-e^{-(\lambda-1)(1-e^{-t})}}{(\lambda-1)(1-e^{-t})}e^{-(1-e^{-t})}\;dt\\
&\arrayeq&\int_0^1 \frac{1-e^{-(\lambda-1)s}}{(\lambda-1)s}e^{-s}\;ds\ \ \ (\mbox{where }s=1-e^{-t})\\
&\arraygeq&\frac{e^{-1}}{\lambda-1}\int_0^{\lambda-1} \frac{1-e^{-u}}{u}\;du\,\asymp\, \frac{\ln\lambda}{\lambda},
\end{eqnarray*}
and
\begin{eqnarray*}
|\Delta^2h(\emptyset;1,1)|&\arrayeq&\int_0^\infty e^{-2t}\mean\frac{1}{\ZZ_\emptyset(t)\{0\}+1}\prob[\ZZ_\emptyset(t)\{1\}=0]\;dt\\
&\arrayeq&\int_0^\infty e^{-2t}\frac{1-e^{-(\lambda-1)(1-e^{-t})}}{(\lambda-1)(1-e^{-t})}e^{-(1-e^{-t})}\;dt\\
&\arrayeq&\int_0^1 (1-s)\frac{1-e^{-(\lambda-1)s}}{(\lambda-1)s}e^{-s}\;ds\ \ \ (\mbox{where }s=1-e^{-t})\\
&\arraygeq&\frac{e^{-1}}{\lambda-1}\int_0^{\lambda-1} \left\{1-\frac{u}{\lambda-1}\right\}\frac{1-e^{-u}}{u}\;du\,\asymp\, \frac{\ln\lambda}{\lambda}.\  \qed
\end{eqnarray*}
\end{example}

As noted before, $\donebar$ is the same as $\done$ when the point patterns have the same number of points while it is smoother than $\done$ when the point patterns do not have the same number of points. On the other hand, for any two point processes $\Xi$ and $\Eta$ on $\mcx$, we have
\begin{equation}\mean d_1(\Xi,\Eta)
=\mean\bigl(d_1(\Xi,\Eta)\bigm| |\Xi|=|\Eta|\bigr)\prob[|\Xi|=|\Eta|]
+\prob[|\Xi|\ne|\Eta|].
\label{xia:comment1}
\end{equation}
When we consider $\prob[|\Xi|\ne|\Eta|]$, which corresponds to the total variation distance between the distributions of the total number of points of the two point processes \citep[see][]{bb92comp}, there is no such logarithmic component in Stein's factor, which means that the logarithmic component in $\done$ was brought in only by the discrepancies of locations of points when the point patterns have the same number of points. However, this problem is shared by $\donebar$, that is, the Stein factors for $\donebar$ will inevitably have the logarithmic component as well.

It is also worthwhile to note that, since $\mean\donebar(\Xi,\Eta)$ replaces the term $\PP[\abs{\Xi} \neq \abs{\Eta}]$ in \Ref{xia:comment1} with a smaller $\EE \dr(\abs{\Xi},\abs{\Eta})$, we would expect a bigger improvement on bounding  $\dtwobar\left(\msl(\Xi),\msl(\Eta)\right)$ when $\PP[\abs{\Xi} \neq \abs{\Eta}]$ is ``dominant'' at the right hand side of \Ref{xia:comment1} under the best coupling. Such an improvement is obtained in the next two subsections.

\subsection{Poisson process approximation of a Bernoulli process}  \label{ssec:BePo}

Let $\mcx=[0,1]$ with $d_0(x,y)=|x-y|$, and let $X_1, \dots, X_n$ be independent and identically distributed Bernoulli random variables with $\prob[X_1=1]=p$. Then $\Xi=\sum_{i=1}^nX_i\delta_{i/n}$ defines a Bernoulli process on $\mcx$. If we let $T_0$, $T_1$, $\dots$, $T_n$ be independent and identically distributed uniform random variables on $\mcx$ which are also independent of $\{X_1,\dots,X_n\}$, then
$$Y=\sum_{i=1}^nX_i\delta_{T_i}$$
defines a binomial process on $\mcx$ \citep[p.~29]{reiss93}. By \citet{xiazhang07},
\begin{equation}
\dtwo(\msl(\Xi),\msl(Y))\le\left(\frac{1}{2n}+\frac{p}{2}\right)\wedge\frac1{\sqrt{3np}}.\label{ZX1}
\end{equation}
To estimate $\dtwobar(\msl(Y),\poisson(\bl))$ with $\bl(dx)=np\, dx$, we employ Stein's method for Poisson process approximation. As a matter of fact, it follows from \Ref{generator1} that
\begin{eqnarray*}
\mean\mca h(Y)&\arrayeq&\mean\left(\int_\mcx\bigl(h(Y+\delta_\alpha)-h(Y)\bigr)\;\bl(d\alpha)+\int_\mcx\bigl(h(Y-\delta_\alpha)-h(Y)\bigr)\;Y(d\alpha)\right)\\
&\arrayeq&np\,\mean\bigl(h(Y+\delta_{T_0})-h(Y)\bigr)+\sum_{i=1}^n\mean\bigl(h(Y^i)-h(Y^i+\delta_{T_i})\bigr)p\\
&\arrayeq&np\,\mean\bigl\{\bigl(h(Y+\delta_{T_0})-h(Y)\bigr)-\bigl(h(Y^1+\delta_{T_0})-h(Y^1)\bigr)\bigr\},
\end{eqnarray*}
where $Y^i=Y-X_i\delta_{T_i}$. Define
$$g(i)=\mean\bigl(h(Y+\delta_{T_0})-h(Y)\bigm|\abs{Y}=i\bigr)=\mean\Bigl(h\bigl(\tsum_{j=0}^i\delta_{T_j}\bigr)-h\bigl(\tsum_{j=1}^i\delta_{T_j}\bigr)\Bigr),$$
then
\begin{eqnarray*}
\left|\mean\mca h(Y)\right|&\arrayeq&np\left|\mean\bigl(g(|Y|)-g(|Y^1|)\bigr)\right|=np^2\left|\mean\bigl(g(|Y^1|+1)-g(|Y^1|)\bigr)\right|
\\
&\arrayleq& 2np^2\|g\|\dtv\bigl(\msl(|Y^1|),\msl(|Y^1|+1)\bigr),\nonumber
\end{eqnarray*}
where $\norm{\cdot}$ denotes the supremum norm and, for any two nonnegative integer-valued random variables $U_1$ and $U_2$,
$$\dtv\bigl(\msl(U_1),\msl(U_2)\bigr):=\frac12 \sup_{\tg:\Zplus \to [-1,1]} \bigabs{\EE \tg(U_1) - \EE \tg(U_2)}.$$

On the other hand, by Lemma~1 in \citet{barbourjensen89}, 
$$\dtv\bigl(\msl(|Y^1|),\msl(|Y^1|+1)\bigr)\le \max_{0\le i\le n-1}\prob[|Y^1|=i]\le 1\wedge\frac{1}{2\sqrt{(n-1)p(1-p)}}
$$
and using \Ref{bounddelta1}, we have, for $f\in\ftwobar$, that 
$$\left|\mean\mca h_f(Y)\right|\le\frac{\left(0.95+\ln^+(np)\right)p}{\frac12\vee\sqrt{(n-1)p(1-p)}},
$$
which implies from \Ref{Steinidentity1} that
\begin{equation}\dtwobar(\msl(Y),\poisson(\bl))=\sup_{f\in\ftwobar}\left|\mean\mca h_{f}(Y)\right|\le\frac{\left(0.95+\ln^+(np)\right)p}{\frac12\vee\sqrt{(n-1)p(1-p)}}.
\label{binomialprocesstopoisoon}
\end{equation}
Now, collecting \Ref{ZX1} and \Ref{binomialprocesstopoisoon} gives
\begin{thm} \label{PoissonBernoullithm} With the above setup, we have
$$\dtwobar(\msl(\Xi),\poisson(\bl))\le \left(\frac{1}{2n}+\frac{p}{2}\right)\wedge\frac1{\sqrt{3np}}+\frac{\left(0.95+\ln^+(np)\right)p}{\frac12\vee\sqrt{(n-1)p(1-p)}}.$$
\end{thm}

\begin{nonumrem}
An immediate message from Theorem~\ref{PoissonBernoullithm} is that, if $n$ is large, it is almost impossible to distinguish between the distributions of the two processes. It is quite a contrast to the conclusion under $\done$ where it is essential to have a very small $p$ as well as a large $n$ to ensure a valid Poisson process approximation (see \citealp{xia97}). In practice, statisticians would use a Poisson process rather than a Bernoulli process when $n$ is large, confirming our conclusion under~$\donebar$.
\end{nonumrem}

\begin{nonumrem}
It is a tantalizing problem to remove the $\ln^{+}\lambda$ term in the upper bound. We conjecture that, at the cost of more complexity, the actual bound should be of order $\left(\frac{1}{n}+p\right)/(1\vee\sqrt{np}).$
\end{nonumrem}

\subsection{Point processes of i.i.d.\ points}  \label{ssec:iidpoints}

Let $\Xi := \sum_{i=1}^M \delta_{X_i}$ and $\Eta := \sum_{i=1}^N \delta_{Y_i}$, where $M$ and $N$ are integer-valued random variables, $(X_i)_{i \in \NN}$ is a sequence of i.i.d.\ $\mcx$-valued random elements that is independent of $M$, and $(Y_i)_{i \in \NN}$ is a sequence of i.i.d.\ $\mcx$-valued random elements that is independent of $N$. 
Denote by $\dw$ the Wasserstein metric between random elements of $\mcx$ with respect to $\dzero$. 
\begin{prop} \label{prop:iidpoints}
   We have
   \begin{equation*}
   \begin{split}
      \max \Bigl( \drw \bigl( \msl(&M), \msl(N) \bigr), c_1 \dw \bigl( \msl(X_1), \msl(Y_1) \bigr) \Bigr) \\
         &\leq \dtwobar \bigl( \msl(\Xi), \msl(\Eta) \bigr)
      \leq \drw \bigl( \msl(M), \msl(N) \bigr) + c_2 \dw \bigl( \msl(X_1), \msl(Y_1) \bigr),
   \end{split}
   \end{equation*}
   where
   \begin{equation*}
      c_1 = c_1 \bigl( \msl(M), \msl(N) \bigr) = {\max \bigl( \PP[M>0],
      \PP[N>0] \bigr)}
   \end{equation*}
   and
   \begin{equation*}
      c_2 = c_2 \bigl( \msl(M), \msl(N) \bigr) = \EE \biggl( \frac{\tM \wedge \tN}{\tM \vee \tN}
         {\bf 1}_{\{\tM \vee \tN > 0\}} \biggr) \leq {\min \bigl( \PP[M>0], \PP[N>0] \bigr)}
   \end{equation*}
   for random variables $\tM$ and $\tN$ that are coupled so that
   $\EE \dr\bigl(\tM,\tN\bigr)=\drw \bigl( \msl(M), \msl(N) \bigr)$.
\end{prop}
\begin{rem}
   An interesting special case is given if $\Xi$ and $\Eta$ are Poisson processes. For finite
   measures $\bm$ and $\bn$ on $\mcx$, we obtain from Proposition~\ref{prop:iidpoints} that
   \begin{equation*}
      \dtwobar \bigl( \poisson(\bm), \poisson(\bn) \bigr) \leq \frac{\abs{\mu - \nu}}{\mu \vee \nu}
      + (1 - e^{- (\mu \wedge \nu)}) \dw \bigl( \bm / \mu, \bn / \nu \bigr),
   \end{equation*}
   which is an improvement by a factor of order $1/\sqrt{\mu \vee \nu}$ for $\mu, \nu \to \infty$ in the 
   first summand when compared to a corresponding $\dtwo$-bound (see for example \citealp{brownxia95},
   Equation~(2.8)).
   Estimation of the $\drw$-term was achieved by considering a Poisson process $Z$ on $\Rplus$ with
   intensity~$1$ and defining a coupling pair by $\tM := Z((0,\mu])$ and $\tN := Z((0,\nu])$.
\end{rem}
\begin{proof}[Proof of Proposition~\ref{prop:iidpoints}]
   {\it Upper bound:}
   Let $\tM \eqinlaw M$ and $\tN \eqinlaw N$ be coupled according to a $\drw$-coupling, so that $\EE 
   \bigl( \frac{\abs{\tM-\tN}}{\tM \vee \tN} {\bf 1}_{\{\tM \vee \tN > 0\}} \bigr) = \drw \bigl( \msl(M), 
   \msl(N) \bigr)$, and let $\tX_i \eqinlaw X_i$ and $\tY_i \eqinlaw Y_i$ with $\EE \dzero(\tX_i, \tY_i) 
   = \dw \bigl( \msl(X_1),\msl(Y_1) \bigr)$ for every $i \in \NN$ in such a way that the pairs 
   $(\tM,\tN)$, $(\tX_1,\tY_1)$, $(\tX_2,\tY_2), \ldots$ are independent. We then obtain
   \begin{equation} \label{eq:iidupper}
   \begin{split}
      \dtwobar \bigl( \msl(\Xi), \msl(\Eta) \bigr) &\leq \EE \donebar \bigl( \tsum_{i=1}^{\tM} \delta_{
         \tX_i}, \, \tsum_{j=1}^{\tN} \delta_{\tY_i} \bigr) \\ 
      &\leq \EE \biggl( \frac{\abs{\tM - \tN}}{\tM \vee \tN} {\bf 1}_{\{\tM \vee \tN > 0\}} \biggr) + \EE 
         \biggl( \frac{{\bf 1}_{\{\tM \vee \tN > 0\}}}{\tM \vee \tN}
         \sum_{i=1}^{\tM \wedge \tN} \dzero(\tX_i, \tY_i) \biggr),
   \end{split}
   \end{equation}
   which, by the independence between $(\tM,\tN)$ and $\{(\tX_i, \tY_i),\ i\ge1\}$, and the assumptions on 
   the distributions of those pairs, yields the upper bound claimed.

   The bound for the factor $c_2$ follows from $\EE \bigl( \frac{\tM \wedge \tN}{\tM \vee \tN} 
   {\bf 1}_{\{\tM \vee \tN > 0\}} \bigr) \leq \PP[\tM>0, \tN>0]$ 
   and $\PP[\tM=\tN=0] = \min \bigl( \PP[M=0], \PP[N=0] \bigr)$, the proof of which is straightforward.
   
   {\it Lower bound:}
   Let $\fwstar := \bigl\{ g : \mcx \to [0,1]; \,
   \abs{g(x) - g(y)} \leq \dzero(x,y) \text{ for all } x,y \in \mcx \bigr\}$, and let $\tg \in \fwstar$
   be a mapping with $\bigabs{ \EE \tg(X_1) - \EE \tg(Y_1) } = \dw \bigl( \msl(X_1), \msl(Y_1) \bigr)$.
   Such a mapping exists by $\dw \bigl( \msl(X_1), \msl(Y_1) 
   \bigr) = \sup_{g \in \fwstar} \bigabs{ \EE g(X_1) - \EE g(Y_1) }$,
   where the supremum is attained because $\fwstar$ is a compact subset of ${\rm C}(\mcx, \RR)$ by the 
   Arzel\`{a}-Ascoli theorem and the mapping $\bigl[ g \mapsto 
   \bigabs{ \EE g(X_1) - \EE g(Y_1) } \bigr]$ is continuous (both statements are with respect to
   the topology of uniform convergence).

   Define $\tf: \mfn \to [0,1]$ by $\tf(\xi) := \frac{1}{\abs{\xi}} \int_{\mcx} \tg(x) \;
   \xi(dx)$ for $\xi \in \mfn \setminus \{0\}$ and $\tf(0) := {\EE \tg(X_1)}$. We next check that $\tf 
   \in \ftwobar$. {It is immediately clear that $\abs{\tf(\xi) - \tf(0)} \leq 1 = \donebar(\xi,0)$ if   
   $\xi \in \mfn \setminus \{0\}$.}
   Let {then $\xi = \sum_{i=1}^m \delta_{x_i}$ and $\eta = \sum_{j=1}^n \delta_{y_j}$ both be in $\mfn 
   \setminus \{0\}$}, where we assume without loss of generality that $m \leq n$ and $\tf(\xi) \geq 
   \tf(\eta)$ (otherwise interchange $\xi$ and $\eta$ and/or replace~$\tg$ by $1-\tg \in \fwstar$),
   and that the points are numbered according to a $\donebar$-pairing  
   such that $\frac{1}{n} \bigl( \sum_{i=1}^m \dzero(x_i,y_i) + (n-m) \bigr) = \donebar(\xi,\eta)$.
   Let $k \in \argmax_{1 \leq i \leq m} \tg(x_i)$, and $x_i := x_k$ for $m+1 \leq i \leq n$, which
   implies
   \begin{equation*}
   \begin{split}
      \bigabs{\tf(\xi) - \tf(\eta)} &= \frac{1}{m} \sum_{i=1}^m \tg(x_i) - \frac{1}{n} \sum_{i=1}^n
         \tg(y_i) \\
      &\leq \frac{1}{n} \sum_{i=1}^n \tg(x_i) - \frac{1}{n} \sum_{i=1}^n
         \tg(y_i) \\
      &\leq \frac{1}{n} \sum_{i=1}^n \dzero(x_i, y_i) \\
      &\leq \donebar(\xi,\eta),
   \end{split}
   \end{equation*}
   and therefore $\tf \in \ftwobar$.

   Choose pairs $(\tM,\tN), (\tX_1,\tY_1), (\tX_2,\tY_2), \ldots$ in the same way as for the proof of 
   the upper bound (although the coupling of $\tX_i$ and $\tY_i$ in each of the pairs is not 
   important now). We obtain
   \begin{equation*}
   \begin{split}
      \dtwobar \bigl( \msl(\Xi), \msl(\Eta) \bigr) &\geq \bigabs{ \EE \tf \bigl( \tsum_{i=1}^{\tM}
         \delta_{\tX_i} \bigr) - \EE \tf \bigl( \tsum_{i=1}^{\tN} \delta_{\tY_i} \bigr) } \\
      &= \biggabs{ \EE \biggl\{ \biggl( \frac{1}{\tM} \sum_{i=1}^{\tM} \tg(\tX_i) - \frac{1}{\tN}
         \sum_{j=1}^{\tN} \tg(\tY_j) \biggr) {\bf 1}_{\{\tM > 0, \tN > 0\}} \biggr\}  \\[-1.5mm]
      &{\hspace*{12mm} {} + \EE \biggl\{ \biggl( \frac{1}{\tM} \sum_{i=1}^{\tM} \tg(\tX_i) - \EE \tg(X_1)
         \biggr) {\bf 1}_{\{\tM > 0, \tN = 0\}} \biggr\}} \\[-1.5mm]
      &{\hspace*{12mm} {} + \EE \biggl\{ \biggl( \EE \tg(X_1) - \frac{1}{\tN} \sum_{j=1}^{\tN}    
         \tg(\tY_j) \biggr) {\bf 1}_{\{\tM = 0, \tN > 0\}} \biggr\}} } \\
      &= \Bigabs{ {\bigl(\EE \tg(X_1) - \EE \tg(Y_1) \bigr) \, \PP[\tM>0, \tN>0]} \\[-1.5mm] 
      &{\hspace*{12mm} {} + \bigl(\EE \tg(X_1) - \EE \tg(Y_1) \bigr) \, \PP[\tM=0, \tN>0]}}\\[1mm]
      &= \dw \bigl( \msl(X_1), \msl(Y_1) \bigr) \, {\PP[N > 0]}.
   \end{split}
   \end{equation*}
   {Since the above argument is symmetric in $\Xi$ and $\Eta$, we obtain the lower bound when combining 
   it with Proposition~\ref{prop:dtwobarprops}(ii).} 
\end{proof}

\section{A statistical application}
\setcounter{figure}{0}

In order to show the potential of $\donebar$ and $\dtwobar$ in statistical applications, we propose a test procedure based on these two metrics. Suppose that our data consists of a few i.i.d.\ realizations of a point process $\Xi$, and we would like to test if $\Xi \sim P$ for a certain probability measure $P$ on $\mfn$. Such multiple point pattern data may arise, among other examples, from recording degenerate cells in tissue samples or plants in a large population that is sampled only via a few comparatively small windows.

In what follows, we restrict our attention to a test for spatial homogeneity under the assumption that $\Xi$ is a Poisson process on $W=[0,1]^2$ with unknown expectation measure $\bl$. This limits the alternative hypothesis sufficiently to keep our simulation study within the scope of this article.
Suppose that $\xi_1, \ldots, \xi_N$ are realizations of i.i.d.\ copies $\Xi_1, \ldots, \Xi_N$ of $\Xi$ and that the total mass $\lambda := \bl([0,1]^2)$ of the expectation measure $\bl$ is known (otherwise we just take the canonical estimate $\frac{1}{N} \sum_{i=1}^N \abs{\xi_i}$). Our null hypothesis is then $\Xi \sim \poisson(\lambda \Leb^2)$. Write $P_N := \frac{1}{N} \sum_{i=1}^N \delta_{\xi_i} \in \mfp(\mfn)$ for the empirical distribution of our data. We perform a Monte Carlo test where the test statistic would ideally be
\begin{equation}
   T(\xi_1,\ldots,\xi_N) := \dtwobar \bigl( P_N, \poisson(\lambda \Leb^2) \bigr),
\end{equation}
but since this is computationally intractable, we replace it by the randomized test statistic
\begin{equation}
   T(\xi_1,\ldots,\xi_N; \eta_1, \ldots, \eta_N) := \dtwobar \bigl( P_N, Q_N \bigr),
\end{equation}
where $Q_N := \frac{1}{N} \sum_{i=1}^N \delta_{\eta_i}$ for realizations $\eta_i$ of $\poisson(\lambda \Leb^2)$-processes $\Eta_i$ that are independent amongst each other and of the $\Xi_i$.
The null hypothesis is rejected at significance level $\alpha=0.05$ if $T(\xi_1,\ldots,\xi_N; \eta_1, \ldots, \eta_N)$ ranks among the five highest values when pooled with 99 simulations of $T(\teta_1,\ldots,\teta_N; \eta_1, \ldots, \eta_N)$, where $\teta_1,\ldots,\teta_N,
\eta_1,\ldots,\eta_N$ are all independent $\poisson(\lambda \Leb^2)$-realizations. 

\begin{figure}[t]
\hspace*{2mm}
\vspace*{-9mm}

\begin{center}
\includegraphics[width=60mm,angle=-90]{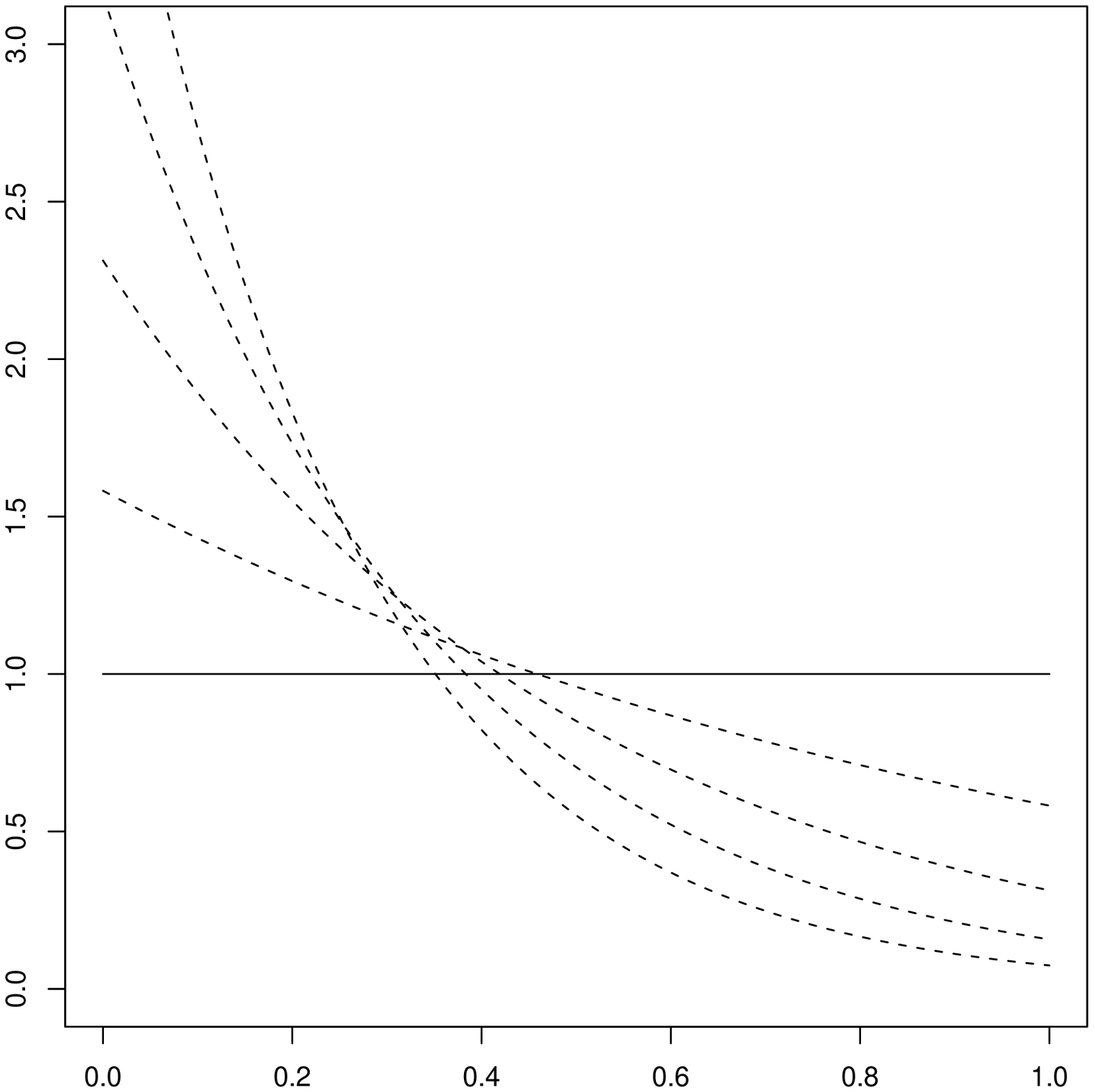} \hspace*{5mm}
\includegraphics[width=60mm,angle=-90]{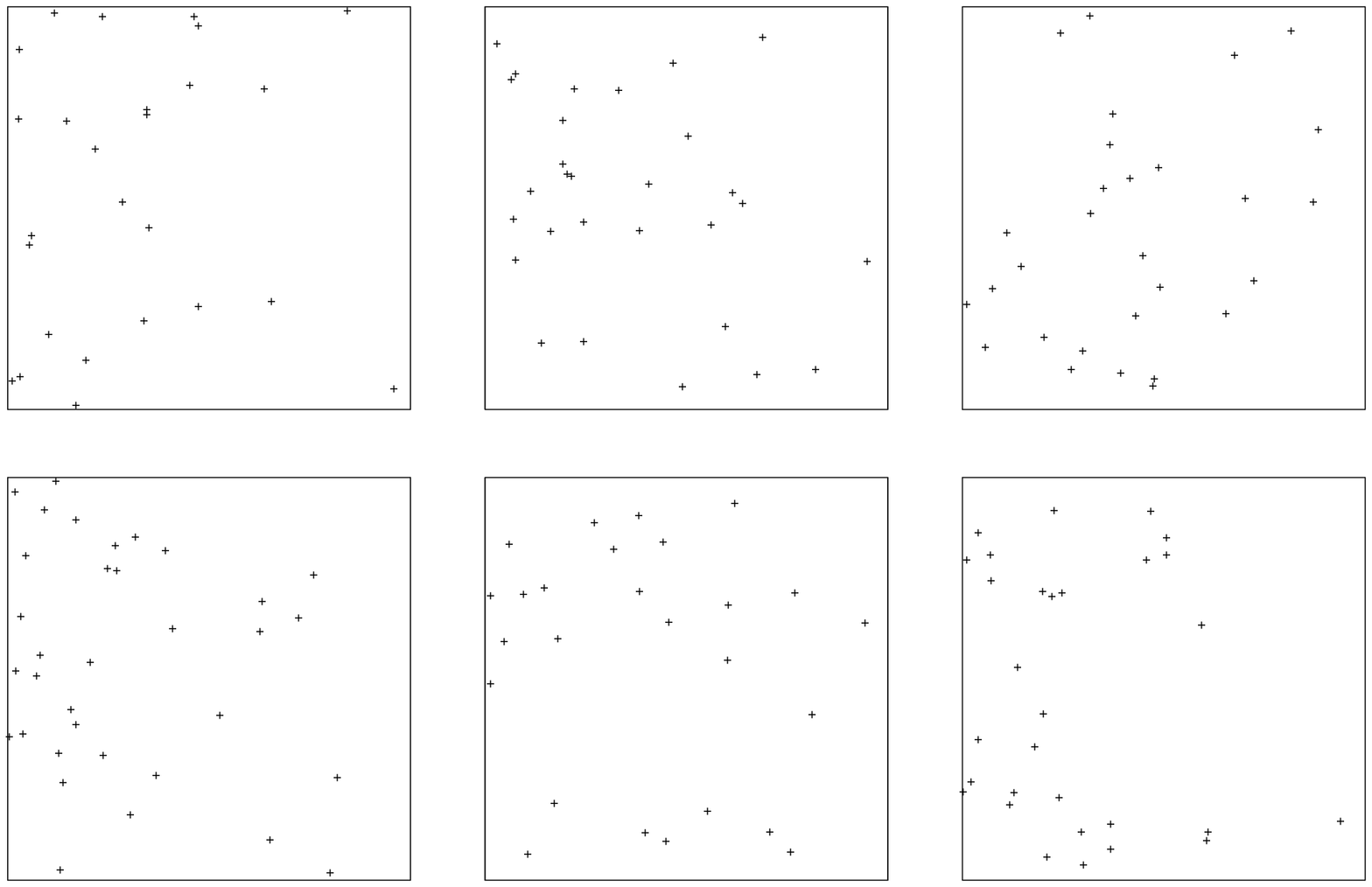} \hspace*{-2mm}
\vspace*{-1mm}

\caption{\small Left hand side: normalized intensity functions $f_{\kappa}$ plotted against their first coordinate; right hand side: six independent realizations from $\poisson(f_2 \Leb^2)$.}
\label{fig:alternatives}
\end{center}
\vspace*{-6mm}

\hspace*{2mm}
\end{figure}

We choose $N=12$, $\lambda=30$ for the simulation study, which is both realistic for actual data and keeps computation times at a tolerable level. One single test of two series of 12 point patterns takes less than three seconds (given the simulated null hypothesis distribution) on an ordinary laptop computer using the library {\sf spatstat} {(see \citealp{spatstat05})} that supplies tools for the analysis of spatial point patterns within the statistical computing environment {\sf R} {(\citealp{r07})}. Increasing either $N$ to $50$ or $\lambda$ to $110$ while keeping the other parameter fixed, still keeps the computation time well under one minute. Note that the optimal point assignments needed for computing $\donebar$, and also $\dtwobar$ between empirical measures, can be found efficiently (in $O \bigl( (m \vee n)^3 \bigr)$ steps, where $m$ and $n$ are the cardinalities of the point patterns) by using the so-called Hungarian method from linear programming \citep[see][Section~11.2]{papasteig98}.

Table~\ref{tab:powers} summarizes the results of our simulations. The first column gives the Monte Carlo powers of our test against $\poisson \bigl( \lambda f_{\kappa}(x,y) \Leb^2(d(x,y)) \bigr)$-alternatives, where
\begin{equation*}
   f_{\kappa}(x,y) = \frac{\kappa \exp(-\kappa x)}{1-\exp(-\kappa)}
\end{equation*} 
for $x,y \in [0,1]$ and $\kappa = 1,\ldots,4$. See Figure~\ref{fig:alternatives} to obtain an impression of the corresponding distributions. By Monte Carlo power we mean the fraction of the number of rejected tests in 100 independent simulations of the alternative. 

For many applications it would be desirable to generalize $\donebar$ by introducing an order parameter $p \geq 1$ and a cut-off value $c>0$, which leads to the definition of
\begin{equation*}
   \donebar^{\hspace*{1.5pt}(p,c)}(\xi,\eta) := \frac{1}{n} \biggl( \min_{\pi \in \Pi_n} \sum_{i=1}^{m}
         \min \bigl( c, \dzero(x_i, y_{\pi(i)}) \bigr)^p + c^{\hspace*{0.75pt}p} (n-m) \biggr)^{1/p}
\end{equation*}
for $\xi = \sum_{i=1}^m \delta_{x_i}, \eta = \sum_{j=1}^n \delta_{y_j} \in \mfn$ with $n \geq \max(m,1)$, where $\dzero$ of course no longer needs to be bounded and is just taken to be the Euclidean metric here.
We stick to the case $p=1$, but give in the second column of Table~\ref{tab:powers} the corresponding Monte Carlo powers if the cut-off is chosen to be $c=0.3$ instead of $1$, so that our test now puts less emphasis on cardinalities and more emphasis on positional differences in the compared point patterns than it did before. There is no strong reason for choosing exactly $c=0.3$; the value reflects the somewhat vague idea that in an optimal pairing of about $30$ points each, the pairing distances are ``usually'' still below 0.3. As one can see from Table~\ref{tab:powers}, the power improvement is very noticeable, and especially this second test detects the inhomogeneity quite well even if they are not very clearly visible by eyeball observation of the simulated data.

For comparison we have also added the results of the corresponding tests if $\donebar$ is replaced by $\done$. Since there is typically a wide range of values for the cardinalities of realizations of a Poisson process with $30$ expected points, and since differing cardinalities are not appropriately addressed by $\done$, these tests perform very poorly (for $c =0.3$, powers seem to lift off from $\kappa=9$ on).

\begin{table}[h]
\begin{center}
\begin{tabular}{|c||c|c||c|c|}
    \hline
    $\kappa$ & \raisebox{0pt}[12pt][0pt]{\makebox[20mm]{$\donebar$, $c=1$}}
             & \makebox[20mm]{$\donebar$, $c=0.3$}
             & \makebox[20mm]{$\done$, $c=1$}
             & \makebox[20mm]{$\done$, $c=0.3$} \\[0.5ex] \hline
             1 & 0.10 & 0.23 & 0.08 & 0.02 \\ 
             2 & 0.41 & 0.97 & 0.12 & 0.06 \\
             3 & 0.93 & 1.00 & 0.06 & 0.04 \\
             4 & 1.00 & 1.00 & 0.10 & 0.10 \\ \hline
\end{tabular}
\caption{\small Powers of the tests for two different cut-off values $c$ against increasingly conspicuous alternatives. The last two columns give the corresponding results when the test is based on the metric $\done$ instead of $\donebar$ and are listed for comparison only.} \label{tab:powers}
\end{center}

\vspace*{-6mm}

\hspace*{2mm}
\end{table}
In summary, the above procedure is rather successful for testing spatial homogeneity from multiple point patterns. We also have obtained promising first results when testing for spatial dependence, but a more extensive further study will be necessary in order to establish the possibilities and limitations of this test procedure and of tests or other statistical analyses based on the $\donebar$-metric in general.

%
%
\section*{Appendix: proofs left out in the main text}

\renewcommand{\thesubsection}{\Alph{section}.\arabic{subsection}}
\renewcommand{\theprop}{\Alph{section}.\Alph{prop}}
\renewcommand{\theequation}{\Alph{section}.\arabic{equation}}
\setcounter{section}{1}
\setcounter{subsection}{0}
\setcounter{prop}{0}
\setcounter{equation}{0}

\begin{proof}[Proof of Proposition~\ref{prop:donebarismetric}]
From the definition it is clear that $0 \leq \donebar(\xi,\eta) \leq 1$, that $\donebar(\xi, \eta) = 0$ if and only if $\xi = \eta$ and that $\donebar(\xi,\eta) = \donebar(\eta,\xi)$. To show the triangle inequality let $\xi = \sum_{i=1}^l \delta_{x_i}, \eta = \sum_{j=1}^m \delta_{y_j}, \zeta = \sum_{k=1}^n \delta_{z_k} \in \mfn$, and add two points $u_1$ and $u_2$ to $\mcx$, extending $\dzero$ by $\dzero(u_1,u_2) := \dzero(u_1,u) := \dzero(u_2,u) := 1$ for every $u \in \mcx$. 

Note that for $l=m=n$ it is straightforward to see that
\begin{equation} \label{eq:samecard}
   \min_{\pi \in \Pi_n} \sum_{i=1}^n \dzero(x_i, y_{\pi(i)}) \leq \min_{\pi \in \Pi_n} \sum_{i=1}^n 
   \dzero(x_i, z_{\pi(i)}) + \min_{\pi \in \Pi_n} \sum_{i=1}^n \dzero(z_i, y_{\pi(i)}),
\end{equation}
which is the essential step in proving the triangle inequality for $\done$.

We now prove that $\donebar(\xi,\eta) \leq \donebar(\xi,\zeta) + \donebar(\zeta,\eta)$, assuming that at most one of the point patterns is empty (otherwise the relation is clearly satisfied). Since this inequality is symmetric in $\xi$ and $\eta$, we assume without loss of generality that $l \leq m$ in what follows. We show two separate cases.

{\it Case~1, $m \leq n$}: Let $x_i := u_1$ for $l+1 \leq i \leq n$ and $y_j := u_2$ for $m+1 \leq j \leq n$, and write $\xibar := \sum_{i=1}^n \delta_{x_i}$ and $\etabar := \sum_{j=1}^n \delta_{y_j}$.
We then have
\begin{equation}
   \donebar(\xi,\eta) \leq \donebar(\xi,\etabar) = \donebar(\xibar,\etabar) \leq \donebar(\xibar,\zeta)
   + \donebar(\zeta,\etabar)   = \donebar(\xi,\zeta) + \donebar(\zeta,\eta),
\end{equation}
using that $a \leq m$ implies $\frac{a}{m} \leq \frac{a+n-m}{n}$ for the first inequality,
and~\eqref{eq:samecard} for the second inequality.

{\it Case~2, $m > n$}:
Let $x_i := z_k := u_1$ for $l+1 \leq i \leq l \vee n$ and $n+1 \leq k \leq l \vee n$, and choose {$y_1', \ldots y_{l \vee n}'$ in such a way that \raisebox{0pt}[12pt][5pt]{$\sum_{j=1}^{l \vee n} \delta_{y_j'} \leq \sum_{j=1}^{m} \delta_{y_j}$} and} $\min_{\pi \in \Pi_{l \vee n}} \sum_{i=1}^{l \vee n} \dzero(z_i, y_{\pi(i)}') = \min_{\pi \in \Pi_m} \sum_{i=1}^{l \vee n} \dzero(z_i, y_{\pi(i)})$.
We have
\begin{equation*}
\begin{split}
   \donebar(\xi,\eta)
   &= \frac{1}{m} \biggl( \min_{\pi \in \Pi_m} \sum_{i=1}^{l \vee n} \dzero(x_i,
      y_{\pi(i)}) + \bigl( m-(l \vee n) \bigr) \biggr) 
      \\
   &\leq \frac{1}{m} \biggl( \min_{\pi \in \Pi_{l \vee n}} \sum_{i=1}^{l \vee n} \dzero(x_i,
      y_{\pi(i)}') + \bigl( m-(l \vee n) \bigr) \biggr) 
      \\ 
   &\leq \frac{1}{m} \biggl( \min_{\pi \in \Pi_{l \vee n}} \sum_{i=1}^{l \vee n} \dzero \bigl( x_i,
      z_{\pi(i)} \bigr) + \min_{\pi \in \Pi_{l \vee n}} \sum_{i=1}^{l \vee n} \dzero \bigl( z_i,
      y_{\pi(i)}' \bigr) + \bigl( m-(l \vee n) \bigr) \biggr) \\
   &= \frac{1}{m} \biggl( \min_{\pi \in \Pi_{l \vee n}} \sum_{i=1}^{l \vee n} \dzero \bigl( x_i,
      z_{\pi(i)} \bigr) + \min_{\pi \in \Pi_m} \sum_{i=1}^{n} \dzero \bigl( z_i,
      y_{\pi(i)} \bigr) + \bigl( m-n \bigr) \biggr) \\[1.5mm]
   &\leq \donebar(\xi,\zeta) + \donebar(\zeta,\eta),
\end{split}
\end{equation*}
where we used \eqref{eq:samecard} for the second inequality.
\end{proof}

\begin{proof}[Proof of Proposition~\ref{prop:donebarprops}]
   Statement~(i) is straightforward from the definitions of $\dr$, $\donebar$ and $\done$.
   
   {\it Statement (ii).} Proposition~4.2 in \cite{xia05} states that $\xi_n \to \xi$ vaguely if and 
      only if $\done(\xi_n, \xi) \to 0$ as $n \to \infty$; so all we need to show is that the latter is 
      equivalent to $\donebar(\xi_n, \xi) \to 0$.
      
      If $\donebar(\xi_n, \xi) \to 0$, we have by (i) that $\dr(\abs{\xi_n},
      \abs{\xi}) \to 0$, from which it is easily seen that $\abs{\xi_n} \to \abs{\xi}$,
      i.e.\ there is an $n_0 \in \NN$ such that $\abs{\xi_n} = \abs{\xi}$ and hence $\donebar(\xi_n, \xi) =
      \done(\xi_n, \xi)$ for every $n \geq n_0$. Thus $\done(\xi_n, \xi) \to 0$.
      The converse direction follows immediately from $\donebar \leq \done$.
      
   {\it Statement (iii).} The local compactness and separability properties depend only on the generated 
      topology. See for example Proposition~4.3 in \cite{xia05} for the proof. Note that, by the 
      compactness of $\mcx$, the sets
      $\mfn_l := \{ \xi \in \mfn; \, \abs{\xi} = l \}$ are compact for all $l \in \Zplus$. 

      It remains to show the completeness. Let $(\xi_n)_{n \in \NN}$ be a $\donebar$-Cauchy sequence in 
      $\mfn$. It is straightforward to see that
      this implies the existence of an $n_0 \in \NN$ such that $\abs{\xi_n} = \abs{\xi_m}$ for every $n,m 
      \geq n_0$, which means that there is an $l \in \Zplus$ such that the tail of $(\xi_n)_{n \in \NN}$ is 
      a Cauchy sequence in $\mfn_l$. By the compactness of $\mfn_l$ this tail converges.
\end{proof}

\begin{proof}[Proof of Proposition~\ref{prop:dtwobarprops}]
   Statement~(i) is an immediate consequence of the Kantorovich-Rubin\-stein theorem, where the minimum is 
   attained, because $(\mfn, \donebar)$ is complete. See \cite{dudley89}, Section 11.8, for details.
   Statement~(ii) follows by taking expectations and minima in Proposition~\ref{prop:donebarprops}(i).
   The last statement follows from \cite{dudley89}, Theorem~11.3.3, using
   Proposition~\ref{prop:donebarprops}(ii) and noting that $\dtwobar$ is an instance of Dudley's
   $\beta$-metric.
\end{proof}

\section*{Acknowledgement}

This work was supported by the ARC Centre of Excellence for Mathematics and Statistics of Complex Systems (AX), the ARC Discovery Grant No.~00530/53001000 (DS), and the Schweizerischer Nationalfonds Fellowship No.~PBZH2-111668 (DS).\\ DS would like to thank Adrian Baddeley for stimulating discussions about spatial statistics.

\bibliographystyle{dcu}

\end{document}